\documentclass[12pt,a4paper]{article}

\skip\footins=8mm plus 1mm

\usepackage{latexsym,amsmath,amsfonts,amssymb,amsthm,enumerate,hyperref}

\usepackage{titlesec}

\titleformat*{\subsection}{\large\bfseries}

\usepackage[labelsep=period]{caption}

\newtheorem{theorem}{Theorem}[section]
\newtheorem{lemma}[theorem]{Lemma}
\newtheorem{corollary}[theorem]{Corollary}
\newtheorem{conjecture}[theorem]{Conjecture}
\theoremstyle{definition}
\newtheorem*{remark}{Remark}
\newtheorem{definition}[theorem]{Definition}

\newtheorem{question}[theorem]{Question}

\newcommand{\Cay}{\mathrm{Cay}}
\newcommand{\id}{\mathrm{id}}
\newcommand{\CC}{\mathcal{C}}

\usepackage{xcolor}
\usepackage[normalem]{ulem}

\begin{document}

\title{Tiling the symmetric group by transpositions}

\author{\renewcommand{\thefootnote}{\arabic{footnote}}Teng Fang\footnotemark[1] , Binzhou Xia\footnotemark[2]}

\footnotetext[1]{Department of Mathematics, Suqian University, Jiangsu 223800, PR China}

\footnotetext[2]{{School} of Mathematics and Statistics, The University of Melbourne, Parkville, VIC 3010, Australia}

\renewcommand{\thefootnote}{}
\footnotetext{{\em E--mail addresses}: \texttt{tfangfm@foxmail.com} (T. Fang), \texttt{binzhoux@unimelb.edu.au} (B. Xia)}

\date{}

\maketitle

\begin{abstract}
For nonempty subsets $X$ and $Y$ of a group $G$, we say that $(X,Y)$ is a tiling of $G$ if every element of $G$ can be uniquely expressed as $xy$ for some $x\in X$ and $y\in Y$. In 1966, Rothaus and Thompson studied whether the symmetric group $S_n$ with $n\geq3$ admits a tiling $(T_n,Y)$, where $T_n$ consists of the identity and all the transpositions in $S_n$. They showed that no such tiling exists if $1+n(n-1)/2$ is divisible by a prime number at least $\sqrt{n}+2$. In this paper, we establish a new necessary condition for the existence of such a tiling: the subset $Y$ must be partition-transitive with respect to certain partitions of $n$. This generalizes the result of Rothaus and Thompson, as well as a result of Nomura in 1985. We also study whether $S_n$ can be tiled by the set $T_n^*$ of all the transpositions, which finally leads us to conjecture that neither $T_n$ nor $T_n^*$ tiles $S_n$ for any $n\geq4$.



\end{abstract}

\section{Introduction}\label{sec:intro}

Let $G$ be a group. A pair $(X,Y)$ of subsets is called a \emph{tiling} of $G$ if every element of $G$ can be written as $xy$ for some unique $x\in X$ and $y\in Y$. We say that $X$ \emph{left-tiles} $G$ if $G$ has a tiling $(X,Y)$ for some subset $Y$. Similarly, $Y$ \emph{right-tiles} $G$ if $G$ has a tiling $(X,Y)$ for some subset $X$. A subset $S$ is said to \emph{tile} $G$ if it either left-tiles or right-tiles $G$. Note that, if $S$ is inverse-closed, then $S$ left-tiles $G$ if and only if $S$ right-tiles $G$.

For subsets $X$ and $Y$ of $G$, it is clear that $(X,Y)$ is a tiling of $G$ if and only if $(xX,Yy)$ is a tiling of $G$ for any $x,y\in G$. Therefore, when studying the tilings $(X,Y)$, one usually assumes that both $X$ and $Y$ contain the identity $\id$ of $G$. We refer to such a tiling as \emph{normalized}. (In some literature, what we call ``normalized tilings'' are simply referred to as ``tilings''.)

Tilings of groups, also referred to as \emph{factorizations} in the literature (see~\cite{SS2009} for example), are closely related to perfect codes in Cayley graphs (see Subsection~\ref{subsec:background}). First studied by Haj\'{o}s~\cite{Hajos1941} in his proof of a well-known conjecture of Minkowski, tilings of abelian groups have received extensive attention~\cite{SS2009}. However, much less is known about tilings of nonabelian groups. In 1960s, Rothaus and Thompson~\cite{RT1966} considered the possible tilings of the symmetric group $S_n$ by the subset
\[
T_n:=\{\id\}\cup\{(i,j)\mid 1\leq i<j\leq n\}
\]
consisting of the identity and all the transpositions in $S_n$. They proved the following result.

\begin{theorem}[Rothaus--Thompson~\cite{RT1966}]\label{RT-Theorem}
If $1+n(n-1)/2$ has a prime factor $p\geq\sqrt{n}+2$, then $T_n$ does not tile $S_n$.
\end{theorem}

The approach of Rothaus and Thompson~\cite{RT1966} to prove the above theorem relies on an ingenious construction of some permutation representation of $S_n$, where the condition ``$1+n(n-1)/2$ has a prime divisor $p\geq\sqrt{n}+2$'' plays an important role. However, following this approach, it seems hard to make further improvements.
In fact, necessary conditions for \( S_n \) to be tiled by \( T_n \) have been investigated in only a few papers~\cite{Etienne1987} and \cite{Nomura1985}. Notably, Nomura~\cite{Nomura1985} in 1985 established an interesting result concerning the multiple transitivity of \( Y \) in any tiling \( (T_n,Y) \) of \( S_n \). A subset \( Y \) of \( S_n \) is said to be \emph{\( k \)-transitive} if there exists a positive integer \( r \) such that, for every pair of \( k \)-tuples \( \alpha \) and \( \beta \) of distinct elements in \(\{1, \dots, n\} \), there are exactly \( r \) permutations in \( Y \) mapping \( \alpha \) to \( \beta \). The main result of~\cite{Nomura1985} is as follows.

\begin{theorem}[Nomura~\cite{Nomura1985}]\label{prop:prop2}
If $(T_n,Y)$ is a tiling of $S_n$, then $Y$ is $k$-transitive for each positive integer $k<n/2$.
\end{theorem}

In the study of the multiple transitivity and homogeneity for groups and sets of permutations, Martin and Sagan \cite{MS2006} made a remarkable generalization of the notion of transitivity which is closely related to the representation theory of $S_n$.
Recall that a sequence $(\lambda_1,\ldots,\lambda_\ell)$ of integers is called a {\em partition} of $n$, denoted $(\lambda_1,\ldots,\lambda_\ell)\vdash n$, if $\lambda_1\geq\dots\geq\lambda_\ell>0$ and $\lambda_1+\cdots+\lambda_\ell=n$.
For convenience, we sometimes treat the sequence $(\lambda_1,\ldots,\lambda_\ell)$ as a multiset and use powers to indicate the multiplicities; for example, $(2,1^3):=(2,1,1,1)$.
An {\em ordered set partition} of $\Omega:=\{1,\ldots,n\}$ is a tuple $P=(P_1,\ldots,P_\ell)$ of pairwise disjoint subsets of $\Omega$ such that $|P_1|\geq\dots\geq|P_\ell|>0$ and $P_1\cup\dots\cup P_\ell=\Omega$. The integer partition $(|P_1|,\cdots,|P_\ell|)\vdash n$ is called the {\em shape} of the ordered set partition $P$.

\begin{definition}[Martin--Sagan \cite{MS2006}]\label{def:partrans}
Let $\lambda$ be a partition of a positive integer $n$. A subset $Y$ of $S_n$ is said to be {\em $\lambda$-transitive} if there exits a positive integer $r$ such that, for each ordered set partitions $P$ and $Q$ of shape $\lambda$, there are exactly $r$ permutations in $Y$ sending $P$ to $Q$.
\end{definition}

In the terminology of Definition~\ref{def:partrans}, the concept of $k$-transitivity is precisely $(n-k, 1^k)$-transitivity. Similarly, $k$-homogeneity is just $(n-k,k)$-transitivity. In this paper, we first reveal the close relationship between tilings of $S_n$ by $T_n$ and the notion of partition-transitivity. This does not only give a necessary condition for a tiling $(T_n,Y)$ of $S_n$ in terms of $\lambda$-transitivity for certain partitions $\lambda$ of $n$ (see Theorem~\ref{main-theorem1}), but also improves both Theorem~\ref{RT-Theorem} and Theorem~\ref{prop:prop2} (see Corollaries~\ref{corol:corol1} and~\ref{corol:corol2}).
We then apply the same approach to the tilings $(T_n^*,Y)$ of $S_n$, where
\[
T_n^*:=T_n\setminus\{\mathrm{id}\}=\{(i,j)\mid 1\leq i<j\leq n\}
\]
is the set of all the transpositions in $S_n$. Based on these, we pose a conjecture with some evidence that neither $T_n$ nor $T_n^*$ tiles $S_n$ for any $n\geq4$ (see Section~\ref{sec:conjecture}).

\subsection{Main results of the paper}\label{sec:MainResult}

For a partition $\lambda=(\lambda_1,\dots,\lambda_\ell)$ of a positive integer $n$, the {\em Young diagram} $D(\lambda)$ of $\lambda$ is an array of $n$ boxes having $\ell$ left-justified rows with row $i$ containing exactly $\lambda_i$ boxes for $i\in\{1,\dots,\ell\}$. In $D(\lambda)$, the {\em content} $\xi(x)$ of a box $x$ in the $i$-th row and $j$-th coloumn is defined by $\xi(x):=j-i$. Our main result is the following theorem and its corollaries.

\begin{theorem}\label{main-theorem1}
If $(T_n,Y)$ is a tiling of $S_n$, then for each $\lambda\vdash n$ with $\sum_{x\in D(\lambda)}\xi(x)\geq0$, the set $Y$ is $\lambda$-transitive.
\end{theorem}

\begin{remark}
One can readily verify that, for $\lambda=(\lambda_1,\dots,\lambda_\ell)\vdash n$, the content sum
\[
\sum_{x\in D(\lambda)}\xi(x)=\sum_{i=1}^\ell\frac{\lambda_i(\lambda_i-2i+1)}{2}.
\]
Thus, $\sum_{x\in D(\lambda)}\xi(x)\geq0$ if and only if $\sum_{i=1}^\ell\lambda_i(\lambda_i-2i+1)\geq0$.
\end{remark}

\begin{corollary}\label{main-corollary1}
If $T_n$ tiles $S_n$, then for each partition $\lambda=(\lambda_1,\dots,\lambda_\ell)$ of $n$ such that $\sum_{i=1}^\ell\lambda_i(\lambda_i-2i+1)\geq0$, the integer $1+n(n-1)/2$ divides $\lambda_1!\cdots\lambda_\ell!$.
\end{corollary}

Theorem~\ref{main-theorem1} also enables us to obtain the following conclusion that slightly strengthens Nomura's result (Theorem~\ref{prop:prop2}).

\begin{corollary}\label{corol:corol2}
If $(T_n,Y)$ is a tiling of $S_n$, then $Y$ is $k$-transitive for each positive integer $k\leq n/2$.
\end{corollary}

As another application of Theorem \ref{main-theorem1}, we are able to slightly weaken the condition $p\geq\sqrt{n}+2$ in Rothaus and Thompson's result (Theorem \ref{RT-Theorem}) to $p\geq\sqrt{n}+1$. This leads to a slightly strengthened version of their result in Corollary~\ref{corol:corol1}.

In~\cite{RT1966}, Rothaus and Thompson briefly remarked that, if $n(n-1)/2$ is divisible by a prime $p\geq\sqrt{n}+2$, then $T_n^*$ does not tile $S_n$. Inspired by this, we establish several results on the tilings of $S_n$ by $T_n^*$ that are parallel to those by $T_n$.

\begin{theorem}\label{main-theorem2}
Suppose that $(T_n^*,Y)$ is a tiling of $S_n$. Then for each partition $\lambda$ of $n$ with $\sum_{x\in D(\lambda)}\xi(x)>0$, the set $Y$ is $\lambda$-transitive. In particular, $Y$ is $k$-transitive for each positive integer $k\leq(n-2)/2$.
\end{theorem}

\begin{corollary}\label{main-corollary2}
If $T_n^*$ tiles $S_n$, then for each partition $\lambda=(\lambda_1,\dots,\lambda_\ell)$ of $n$ such that $\sum_{i=1}^\ell\lambda_i(\lambda_i-2i+1)>0$, the integer $n(n-1)/2$ divides $\lambda_1!\cdots\lambda_\ell!$.
\end{corollary}

Based on the above results, we believe that neither $T_n$ nor $T_n^*$ tiles $S_n$ for any $n\geq4$ (see Conjecture~\ref{conj:conj1}), and various attempts are made in Section~\ref{sec:conjecture} towards this conjecture.

\subsection{More background}\label{subsec:background}

The problem whether $T_n$ tiles $S_n$ is a long-standing open problem. In~\cite{Schmidt2002}, it is referred to as ``the $S_n$ problem''.
Rothaus and Thompson~\cite{RT1966} used the terminology ``divide'' instead of ``tile''. Diaconis~\cite[Pages~45--46]{Diaconis1989} linked the Rothaus--Thompson approach~\cite{RT1966} to the study of random walks on $S_n$ and gave the coding-theoretic background of the $S_n$ problem, as illustrated below.

In a simple graph $\Gamma$, a {\em perfect code} \cite{Biggs1973} is a set $C$ of vertices such that every vertex of $\Gamma$ is at distance at most one to exactly one vertex in $C$.
Given a group $G$ and a subset $S$ of $G$, the {\em Cayley digraph} $\Cay(G,S)$ is the digraph with vertex set $G$ and arcs $(g,h)$ whenever $hg^{-1} \in S$. Clearly, $\Cay(G,S)$ is undirected if and only $S$ is inverse-closed, and $\Cay(G,S)$ is loop-free if and only if $\mathrm{id}\notin S$.
It is readily seen that a subset $Y$ of $G$ is a {perfect code} in $\Cay(G,S)$ if and only if every element of $G$ can be written uniquely as a product $xy$ with $x\in S\cup\{\mathrm{id}\}$ and $y\in Y$, or equivalently, if and only if $(S\cup\{\mathrm{id}\},Y)$ is a tiling of $G$. Therefore, the $S_n$ problem is precisely concerned with the existence of perfect codes in $\Cay(S_n,T_n^*)$.

Similarly, a subset $C$ of vertices is said to be a {\em total perfect code}~\cite{Zhou2016} in a simple graph $\Gamma$ if every vertex of $\Gamma$ has exactly one neighbor in $C$. Hence, $(T_n^*,Y)$ is a tiling of $S_n$ if and only if $Y$ is a total perfect code in $\Cay(S_n,T_n^*)$. In graph theory, a perfect code is also called an \emph{efficient dominating set} \cite{DS2003} or {\em independent perfect dominating set}~\cite{Lee2001}, and a total perfect code is called an {\em efficient open dominating set}~\cite{HHS1998}.

The problem of tiling $S_n$ by certain subsets of $T_n$ is considered in~\cite{EW1990}. As a natural generalization of tilings, an {\em $r$-tiling} of a group $G$ is a pair $(X,Y)$ of subsets of $G$ such that every element of $G$ can be written precisely in $r$ ways as $xy$ for some $x\in X$ and $y\in Y$. $r$-Tilings $(X,Y)$ of $S_n$ and $\mathrm{SL}_2(q)$ such that $X$ is closed under conjugation are studied in~\cite{Fang,FLZ2021,GL2020,Terada2004}.

\subsection{Structure of the paper}

After this introduction, we assemble in Section \ref{sec:prelim} preliminary results on the representations of symmetric groups. Section~\ref{sec:proofofmainresults} starts with some technical lemmas for tilings of $S_n$ in the language of group algebra and is followed by two more subsections. We tie together all the preparation to prove Theorem~\ref{main-theorem1} in Subsection~\ref{subsec:proofofmaintheorem}. Then in Subsection~\ref{subsec:corols} some corollaries and related results
of Theorem~\ref{main-theorem1} are obtained. Finally, we pose the conjecture on the tilings of $S_n$ by transpositions, with various attempts, in Section~\ref{sec:conjecture}.

\section{Representation-theoretic results}\label{sec:prelim}

In this section, we collect some notions and results from representation theory that will be needed in the paper.

\subsection{Representations of finite groups}

Let $\mathbb{C}G$ be the group algebra of a finite group $G$ over $\mathbb{C}$, and identify the representations of $G$ (over $\mathbb{C}$) as $\mathbb{C}G$-modules. A submodule of the left regular representation of $G$ is exactly a left ideal of $\mathbb{C}G$, and an irreducible $\mathbb{C}$-linear representation of $G$ is equivalent to a minimal left ideal of $\mathbb{C}G$. By Maschke's theorem, each left ideal of $\mathbb{C}G$ is a direct sum of minimal left ideals of $\mathbb{C}G$.

For a character $\chi$ of a representation of $G$, we naturally ($\mathbb{C}$-linearly) extend it to $\chi\colon\mathbb{C}G\rightarrow\mathbb{C}$, and denote
\begin{equation}\label{equ:3}
c_\chi:=\dfrac{\chi(\mathrm{id})}{|G|}\sum\limits_{g\in G}\chi(g^{-1})g\ \text{ and }\  I_\chi:=(\mathbb{C}G)c_\chi.
\end{equation}
The results in the following lemma are well known and can be read off from~\cite[\S33]{CR1962}.

\begin{lemma}\label{lem:decomp}
Let $\chi_1,\dots,\chi_r$ be a complete set of irreducible characters of $G$. Then the following statements hold:
\begin{enumerate}[{\rm(a)}]
\item $c_{\chi_1},\dots,c_{\chi_r}$ form a $\mathbb{C}$-basis for the center of $\mathbb{C}G$.
\item For each $i\in\{1,\dots,r\}$, the element $c_{\chi_i}$ is the identity of $I_{\chi_i}$ and annihilates $I_{\chi_j}$ for all $j\in\{1,\dots,r\}\setminus\{i\}$. In particular, each $c_{\chi_i}$ is a central idempotent of $\mathbb{C}G$, and $I_{\chi_i}I_{\chi_j}=I_{\chi_j}I_{\chi_i}=0$ for distinct $i$ and $j$ in $\{1,\dots,r\}$.
\item $1=\sum_{i=1}^rc_{\chi_i}$, and hence $\mathbb{C}G=\bigoplus_{i=1}^rI_{\chi_i}$ is a decomposition of $\mathbb{C}G$ into simple two-sided ideals with the multiplication by $c_{\chi_i}$ being the projection into $I_{\chi_i}$.
\item For each $i\in\{1,\dots,r\}$, the two-sided ideal $I_{\chi_i}$ is the sum of all the minimal left ideals of $\mathbb{C}G$ corresponding to the character $\chi_i$.
\end{enumerate}
\end{lemma}

\subsection{Representations of $S_n$}\label{subsec:repofSn}

Let $\lambda=(\lambda_1,\ldots,\lambda_\ell)$ and $\mu=(\mu_1,\ldots,\mu_m)$ be partitions of $n$. We say that $\mu$ {\em dominates} $\lambda$, written $\mu\unrhd\lambda$, if
\[
\mu_1+\mu_2+\cdots+\mu_i\geq\lambda_1+\lambda_2+\cdots+\lambda_i
\]
for each positive integer $i$, where, if $i>\ell$ (respectively, $i>m$), then we take $\lambda_i$ (respectively, $\mu_i$) to be $0$.
A {\em Young tableau of shape $\lambda$} (also called a {\em $\lambda$-tableau}) is an array obtained by replacing the boxes of the Young diagram of $\lambda$ with the numbers $1,2,\ldots,n$ bijectively. For a $\lambda$-tableau $t$, the subgroup $H(t)$ of $S_n$ fixing each row of $t$ is called the {\em horizontal group} ({\em row group}) of $t$, and the subgroup $V(t)$ of $S_n$ fixing each column of $t$ is called the {\em vertical group} ({\em column group}) of $t$.

Let $\lambda=(\lambda_1,\ldots,\lambda_\ell)$ be a partition of $n$. Two $\lambda$-tableaux $s$ and $t$ are said to be {\em row-equivalent}, written $s\sim t$, if for each $i\in\{1,2,\ldots,\ell\}$, the sets of numbers in the $i$-th row of $s$ and $t$ are equal. A {\em $\lambda$-tabloid} is a row-equivalence class of $\lambda$-tableaux, and the $\mathbb{C}S_n$-module $M^\lambda$ corresponding to $\lambda$ is the $\mathbb{C}$-vector space with the $\lambda$-tabloids as basis and the natural permutation action of $S_n$ on $\lambda$-tabloids. For a $\lambda$-tableau $t$, we denote the tabloid $\{s\,|\,s\sim t\}$ by $[t]$ and call
\[
e_t:=\sum_{\sigma\in V(t)}\mathrm{sgn}(\sigma)\sigma[t]
\]
a {\em $\lambda$-polytabloid}. The {\em Specht module $S^\lambda$} is the submodule of $M^\lambda$ spanned by all $\lambda$-polytabloids. The results in the following lemma can be found in~\cite[\S2.4 and \S2.11]{Sagan2001} and~\cite[Corollary~2.2.22]{JK1981}.

\begin{lemma}\label{lem:specht}
The following statements hold:
\begin{enumerate}[{\rm(a)}]
\item The set of Specht modules $\{S^\lambda\}_{\lambda\vdash n}$ forms a complete set of pairwise inequivalent irreducible $\mathbb{C}S_n$-modules.

\item For each $\lambda\vdash n$, the $\mathbb{C}S_n$-module $M^\lambda$ decomposes as
\[
M^\lambda\,\cong\,\bigoplus_{\mu\,\unrhd\, \lambda} K_{\mu\lambda}S^\mu,
\]
where $K_{\mu\lambda}$ is a nonnegative integer, known as the Kostka number, such that $K_{\lambda\lambda}=1$ and that $K_{\mu\lambda}\geq 1$ if and only if $\mu\unrhd\lambda$.
\end{enumerate}

\end{lemma}

In what follows, we will write $\chi^\lambda$ for the {\em character} of the irreducible representation $S^\lambda$ and write $\chi^\lambda(\mu)$ or $\chi^\lambda_\mu$ for the character value of $\chi^\lambda$ at a conjugacy class with cycle type $\mu\vdash n$.
With the notation in~\eqref{equ:3}, we write
\begin{equation}\label{equ:4}
c^\lambda:=c_{\chi^\lambda}\ \text{ and }\  I^\lambda:=I_{\chi^\lambda}.
\end{equation}
For a subset $A$ of $S_n$, denote
\[
\overline{A}:=\sum_{a\in A}a \in\mathbb{C}S_n.
\]
Since there is only one $(n)$-tabloid, the $\mathbb{C}S_n$-module $M^{(n)}$ is trivial, and so is $S^{(n)}$ by Lemma~\ref{lem:specht}. Therefore,
\begin{equation}\label{equ:6}
c^{(n)}=\frac{1}{n!}\overline{S_n}\ \text{ and }\ I^{(n)}=\mathbb{C}\overline{S_n}.
\end{equation}
Let $\CC_\mu$ denote the conjugacy class in $S_n$ with cycle type $\mu$. We will also need the notion of the {\em central character}
\begin{equation}\label{equ:5}
\omega^\lambda_\mu:=\frac{|\CC_\mu|\chi^\lambda_\mu}{\chi^\lambda(\mathrm{id})},
\end{equation}
where $\chi^\lambda(\mathrm{id})=\chi^\lambda_{(1^n)}$ is the degree of $\chi^\lambda$. Recall from Subsection~\ref{sec:MainResult} the definition of the content $\xi(x)$ of a box $x$ in a Young diagram. Interestingly, the central character $\omega^{\lambda}_{(2,1^{n-2})}$ is equal to the content sum $\sum_{x\in D(\lambda)}\xi(x)$~\cite{Frobenius1901,Ingram1950} (see also~\cite{CGS2004}) and is monotonic with respect to the dominance order~\cite[Lemma 10]{DS1981}. We record these results in the next lemma.

\begin{lemma}\label{lem:cencha}
Let $\lambda$ and $\mu$ be partitions of $n$. Then the following statements hold:
\begin{enumerate}[{\rm(a)}]
\item $\omega^{\lambda}_{(2,1^{n-2})}=\sum_{x\in D(\lambda)}\xi(x)$.
\item If $\mu\unrhd\lambda$ and $\mu\neq\lambda$, then $\omega^\mu_{(2,1^{n-2})}>\omega^\lambda_{(2,1^{n-2})}$.
\end{enumerate}
\end{lemma}

Both $M^\lambda$ and $S^\lambda$ have their left ideal avatars in the group algebra $\mathbb{C}S_n$, which we now describe.
Let $\lambda\vdash n$, and let $t$ be a $\lambda$-tableau. For a subgroup $H$ of $S_n$, the mapping $w\overline{H}\mapsto wH$ (this mapping is the $\mathbb{C}$-linear span of the mapping $g \overline{H}\mapsto gH$ for each $g\in S_n$) is a $\mathbb{C} S_n$-module isomorphism from $(\mathbb{C}S_n) \overline{H}$ to the permutation representation of $S_n$ by left multiplication on the set $S_n/H$ of the left cosets of $H$. Hence, $(\mathbb{C}S_n) \overline{H(t)}$ is isomorphic to $M^\lambda$ by~\cite[Theorem 2.1.12]{Sagan2001}. Denote
\[
E_t:=\sum_{\pi\in V(t)}\sum_{\rho\in H(t)}\mathrm{sgn}(\pi)\pi\rho\ \text{ and }\  L^\lambda:=(\mathbb{C}S_n)E_t.
\]
It turns out that $L^\lambda$ is an avatar of $S^\lambda$ in $\mathbb{C}S_n$ (see for instance~\cite[Lemma 7.1.4]{JK1981}). Therefore, we have the following lemma.

\begin{lemma}\label{lem:spechtavatar}
Let $\lambda\vdash n$, and let $t$ be a $\lambda$-tableau. Then there hold the following isomorphisms of $\mathbb{C}S_n$-modules:
\begin{enumerate}[{\rm(a)}]
\item The left ideal $(\mathbb{C}S_n)\overline{H(t)}$ of $\mathbb{C}S_n$ is isomorphic to $M^\lambda$.
\item The left ideal $L^\lambda$ of $\mathbb{C}S_n$ is isomorphic to $S^\lambda$.
\end{enumerate}
\end{lemma}

\section{$\lambda$-transitivity and tilings of $S_n$ by transpositions}\label{sec:proofofmainresults}

In this section, we prove Theorem~\ref{main-theorem1}--Corollary~\ref{main-corollary2}, as well as a slightly strengthened version (see~Corollary~\ref{corol:corol1}) of Rothaus--Thompson's result (Theorem~\ref{RT-Theorem}).

\subsection{Technical lemmas}

By definition, $(X,Y)$ is a tiling of a group $G$ if and only if for each $g\in G$ there exists a unique $(x,y)\in X\times Y$ satisfying $g=xy$, which is equivalent to the equation $\overline{G}=\overline{X}\,\overline{Y}$ in $\mathbb{C}G$.
For a subset $X$ that is closed under conjugation in $G$, the pair $(X,Y)$ is a tiling of $G$ if and only if $(X,gYh)$ is a tiling of $G$ for any $g,h\in G$. Moreover, if $X$ is closed under conjugation in $G$, then $\overline{X}$ is central in $\mathbb{C}G$ and so $\overline{X}\,\overline{Y}=\overline{Y}\,\overline{X}$. In this case, if $X$ is inverse-closed in addition, then $(X,Y)$ being a tiling of $G$ is equivalent to $(X,Y^{-1})$ being a tiling of $G$, where $Y^{-1}:=\{y^{-1}\,|\,y\in Y\}$. We summarize these observations in the following lemma.

\begin{lemma}\label{lem:tiling}
Let $X$ and $Y$ be subsets of a group $G$. Then the following statements hold:
\begin{enumerate}[{\rm(a)}]
\item $(X,Y)$ is a tiling of $G$ if and only if $\overline{G}=\overline{X}\,\overline{Y}$.
\item If $(X,Y)$ is a tiling of $G$ such that $X$ is closed under conjugation in $G$, then $(X,gYh)$ is also a tiling of $G$ for all $g,h\in G$.
\item If $(X,Y)$ is a tiling of $G$ such that $X$ is closed under inversion and conjugation in $G$, then $(X,Y^{-1})$ is also a tiling of $G$.
\end{enumerate}
\end{lemma}

Since $T_n$ is closed under conjugation in $S_n$, the element $\overline{T_n}$ of $\mathbb{C} S_n$ is central.
Recall the definition of $c^\lambda$ and $I^\lambda$ in~\eqref{equ:4}, where $\lambda$ is a partition of $n$, and recall the definition of the central character in~\eqref{equ:5}. The next lemma shows that the multiplication by $\overline{T_n}$ restricted on $I^\lambda$ is the scalar multiplication  of $1+\omega^\lambda_{(2,1^{n-2})}$.

\begin{lemma}\label{lem:conjusum}
Let $\lambda\vdash n$. Then $\overline{T_n}\,u=u\,\overline{T_n}=\big(1+\omega^\lambda_{(2,1^{n-2})}\big)u$ for all $u\in I^\lambda$.
\end{lemma}

\begin{proof}
Consider the minimal left ideal $L^\lambda$ (see Lemmas~\ref{lem:specht} and~\ref{lem:spechtavatar}) and the mapping
\[
\varphi\colon L^\lambda\rightarrow L^\lambda,\ \ u\mapsto \overline{T_n}\,u.
\]
Since $\overline{T_n}$ lies in the center of $\mathbb{C}S_n$, then by Schur's lemma, there exists some constant $d\in\mathbb{C}$ such that $\varphi(u)=du$ for all $u\in L^\lambda$. Hence we have
\[
\chi^\lambda(\mathrm{id})d=\mathrm{Tr}(\varphi)=\chi^\lambda(\overline{T_n})
=\chi^\lambda(\mathrm{id})+|\CC_{(2,1^{n-2})}|\chi^\lambda_{(2,1^{n-2})}
=\chi^\lambda(\mathrm{id})+\chi^\lambda(\mathrm{id})\omega^\lambda_{(2,1^{n-2})},
\]
where $\mathrm{Tr}(\varphi)$ is the trace of $\varphi$. This leads to $d=1+\omega^\lambda_{(2,1^{n-2})}$. In other words,
$\overline{T_n}\,u=\big(1+\omega^\lambda_{(2,1^{n-2})}\big)u$ for all $u\in L^\lambda$.
Note from Lemma~\ref{lem:decomp} that $I^\lambda$ is the sum of all the minimal left ideals of $\mathbb{C}G$ corresponding to the character $\chi^\lambda$. We conclude that $\overline{T_n}\,u=\big(1+\omega^\lambda_{(2,1^{n-2})}\big)u$ for all $u\in I^\lambda$, as required.
\qed
\end{proof}

\medskip
Based on the above discussion, we obtain the following necessary condition for $(T_n,Y)$ to be a tiling of $S_n$.

\begin{lemma}\label{lem:necessary}
Suppose that $(T_n,Y)$ is a tiling of $S_n$ with $n\geq3$. Then $c^\lambda\overline{Y}=0$ for each $\lambda\vdash n$ such that $\lambda\neq (n)$ and $\omega^\lambda_{(2,1^{n-2})}\neq -1$, or equivalently, $\overline{Y}\in I^{(n)}\boldsymbol{\oplus}\bigg(\bigoplus_{\lambda\in\Lambda}I^\lambda\bigg)$, where $\Lambda:=\{\lambda\vdash n\mid\omega^\lambda_{(2,1^{n-2})}=-1\}$.
\end{lemma}

\begin{proof}
Let $\lambda\vdash n$ such that $\lambda\neq (n)$ and $\omega^\lambda_{(2,1^{n-2})}\neq-1$. Note from~\eqref{equ:6} that $I^{(n)}=\mathbb{C}\overline{S_n}$. Then $c^\lambda\overline{S_n}=0$ by Lemma~\ref{lem:decomp}, and $c^\lambda\,\overline{T_n}=\big(1+\omega^\lambda_{(2,1^{n-2})}\big)c^\lambda$ by Lemma \ref{lem:conjusum}.
Since $(T_n,Y)$ is a tiling of $S_n$, it follows from Lemma~\ref{lem:tiling} that $\overline{T_n}\,\overline{Y}=\overline{S_n}$, which yields
\[
0=c^\lambda \overline{S_n}=c^\lambda\overline{T_n}\,\overline{Y}=\big(1+\omega^\lambda_{(2,1^{n-2})}\big)(c^\lambda\overline{Y}).
\]
Thus, $c^\lambda\overline{Y}=0$, as required. To see the equivalent conclusion, notice Lemma~\ref{lem:decomp} and that $(n)\notin\Lambda$ as $\omega^{(n)}_{(2,1^{n-2})}=n(n-1)/2$ by Lemma~\ref{lem:cencha}.
\qed
\end{proof}

\subsection{Proof of Theorem \ref{main-theorem1} and Corollary \ref{main-corollary1}}\label{subsec:proofofmaintheorem}

With the preparation so far, we now embark on the proof of Theorem~\ref{main-theorem1} and Corollary~\ref{main-corollary1}. They hold trivially for $n=2$. Hence we assume $n\geq3$ in the following. Let $(T_n,Y)$ be a tiling of $S_n$, and let $\lambda$ be a partition of $n$ such that $\sum_{x\in D(\lambda)}\xi(x)\geq0$.

Recall from Subsection \ref{subsec:repofSn} the definition of the $\mathbb{C}S_n$-module $M^\lambda$. By Lemmas \ref{lem:specht} and \ref{lem:spechtavatar}, there exists a $\mathbb{C}S_n$-isomorphism $\rho$ from $M^\lambda$ to $\bigoplus_{\mu\,\unrhd\, \lambda} K_{\mu\lambda}L^\mu$. Let $[s]$ and $[t]$ be arbitrary $\lambda$-tabloids. Then $h[s]=[t]$ for some $h\in S_n$, and so we derive from~\eqref{equ:6} that
\begin{align*}
c^{(n)}\rho([s]-[t])
&=c^{(n)}\big(\rho([s])-h\rho([s])\big)\\
&=\frac{1}{n!}\,\overline{S_n}\big(\rho([s])-h\rho([s])\big)
=\frac{1}{n!}\,\overline{S_n}\,\rho([s])-\frac{1}{n!}\,\overline{S_n}\,\rho([s])=0.
\end{align*}
This together with Lemma~\ref{lem:decomp} implies that
\begin{equation}\label{equ:1}
\rho([s]-[t])\in\bigoplus_{\substack{\mu\,\unrhd\, \lambda\\ \mu\neq (n)}} K_{\mu\lambda}L^\mu.
\end{equation}
For each $\mu\unrhd\lambda$ with $\mu\neq (n)$, we have by Lemma \ref{lem:cencha} that $\omega^\mu_ {(2,1^{n-2})}\geq\omega^\lambda_{(2,1^{n-2})}\geq0$ (in particular, $\omega^\mu_{(2,1^{n-2})}\neq -1$), and thus $\overline{Y}L^\mu=0$ by Lemma~\ref{lem:decomp}. It then follows from~\eqref{equ:1} that $\overline{Y}\rho([s]-[t])=0$, which is equivalent to $\overline{Y}([s]-[t])=0$. According to Lemma~\ref{lem:tiling}, $(T_n,Y^{-1})$ is also a tiling of $S_n$, whence we have $\overline{Y^{-1}}([s]-[t])=0$ in addition. Therefore,
\begin{equation}\label{equ:2}
\overline{Y}\,[s]=\overline{Y}\,[t]\ \text{ and }\ \overline{Y^{-1}}\,[s]=\overline{Y^{-1}}\,[t].
\end{equation}

Now fix a $\lambda$-tabloid $[s_0]$, and let $r$ be the numbers of elements $y$ in $Y$ such that $y^{-1}[s_0]=[s_0]$. For arbitrary $\lambda$-tabloids [$s_1$] and [$s_2$], it follows from~\eqref{equ:2} that $\overline{Y}\,[s_1]=\overline{Y}\,[s_0]$ and $\overline{Y^{-1}}\,[s_2]=\overline{Y^{-1}}\,[s_0]$.
Comparing the coefficients of $[s_0]$ in both sides of $\overline{Y^{-1}}\,[s_2]=\overline{Y^{-1}}\,[s_0]$, we obtain
\[
|\{y\in Y\mid y^{-1}[s_2]=[s_0]\}|=|\{y\in Y\mid y^{-1}[s_0]=[s_0]\}|=r.
\]
Then comparing the coefficients of $[s_2]$ in both sides of $\overline{Y}\,[s_1]=\overline{Y}\,[s_0]$, we infer that
\begin{align*}
|\{y\in Y\mid y[s_1]=[s_2]\}|&=|\{y\in Y\mid y[s_0]=[s_2]\}|\\
&=|\{y\in Y\mid y^{-1}[s_2]=[s_0]\}|=r.
\end{align*}
By the definition of $\lambda$-transitivity and $M^\lambda$, this means that for each ordered set partitions $P$ and $Q$ of shape $\lambda$, there are exactly $r$ permutations in $Y$ sending $P$ to $Q$.
Write $\lambda=(\lambda_1,\ldots,\lambda_\ell)$. Then there are precisely $n!/\lambda_1 !\cdots\lambda_\ell !$ ordered set partitions of shape $\lambda$, and so $r=|Y|/(n!/\lambda_1 !\cdots\lambda_\ell !)>0$. Thus, according to the definition of $\lambda$-transitivity, $Y$ is $\lambda$-transitive, as Theorem~\ref{main-theorem1} asserts. Moreover, since $r$ is an integer, $n!/\lambda_1 !\cdots\lambda_\ell !$ divides $|Y|=|S_n|/|T_n|=n!/|T_n|$. It follows that $1+n(n-1)/2=|T_n|$ divides $\lambda_1 !\cdots\lambda_\ell !$, which together with the Remark after Theorem~\ref{main-theorem1} implies Corollary~\ref{main-corollary1}.
\qed

\subsection{Other results}\label{subsec:corols}

\begin{corollary}[{A strengthened version of \cite[Theorem]{RT1966}}]\label{corol:corol1}
If $1+n(n-1)/2$ has a prime divisor $p\geq\sqrt{n}+1$, then $T_n$ does not tile $S_n$.
\end{corollary}

\begin{proof}
Suppose on the contrary that there exists some tiling $(T_n,Y)$ of $S_n$. Write $n=(p-1)q+r$ with nonnegative integers $q$ and $r$ such that $r\leq p-2$. Consider the partition $\lambda=(\lambda_1,\ldots,\lambda_\ell)=((p-1)^q,r)\vdash n$. Then
\begin{align}
\sum_{i=1}^\ell\lambda_i(\lambda_i-2i+1)&=\sum_{i=1}^q(p-1)(p-2i)+r(r-2q-1)\nonumber\\
&=(p-1)(p-1-q)q+r(r-2q-1).\label{equ:7}
\end{align}
By Corollary~\ref{main-corollary1}, it suffices to prove that~\eqref{equ:7} is nonnegative.
Since $n=(p-1)q+r\geq(p-1)q$ and $p\geq\sqrt{n}+1$, we have
\[
q\leq\frac{n}{p-1}\leq\frac{n}{\sqrt{n}}=\sqrt{n}\leq p-1.
\]
If $q=p-1$, then $\sqrt{n}=p-1=q$ and hence $r=n-(p-1)q=0$, which yields that~\eqref{equ:7} equals $0$. Now assume that $q\leq p-2$.
Then since the minimum of $r(r-2q-1)$ over integers $r$ is $-(q+1)q$, it follows that~\eqref{equ:7} is at least
\[
(p-1)(p-1-q)q-(q+1)q\geq\big((q+2)-1\big)\big((q+2)-1-q\big)q-(q+1)q=0.
\]
This completes the proof.
\qed
\end{proof}

\medskip
Next we give a \textbf{proof of Corollary~\ref{corol:corol2}}:

\smallskip\noindent
Consider the hook-like partition $\lambda=(n-k,1^k)\vdash n$. It satisfies $\sum_{x\in D(\lambda)}\xi(x)\geq0$ if and only if $n-k\geq k+1$. If $k<n/2$, then $Y$ is $\lambda$-transitive by Theorem~\ref{main-theorem1}. Thus we only need to deal with the case when $n=2k\geq4$ is even. As in the proof of Theorem~\ref{main-theorem1}, it suffices to show that in the decomposition $M^\lambda\cong \bigoplus_{\mu\,\unrhd\, \lambda} K_{\mu\lambda}L^\mu$ we have $\omega^\mu_{(2,1^{n-2})}\neq -1$ whenever $\mu\unrhd\lambda$. Clearly, $\omega^\lambda_{(2,1^{n-2})}=-k\neq-1$. If $\mu\unrhd\lambda$ and $\mu\neq\lambda$, then $\mu\unrhd (k,2,1^{k-2})$ and hence we deduce from Lemma~\ref{lem:cencha} that $\omega^\mu_{(2,1^{n-2})}\geq\omega^{(k,2,1^{k-2})}_{(2,1^{n-2})}=0$.
\qed

\medskip
Now consider the tilings of $S_n$ by $T_n^*$. Through a very similar argument as the above proof of Theorem~\ref{main-theorem1} and its corollaries, we give a \textbf{proof of Theorem~\ref{main-theorem2} and Corollary~\ref{main-corollary2}}:

\smallskip\noindent
By Lemma~\ref{lem:conjusum}, $\overline{T_n^*}\,u=u\,\overline{T_n^*}=\omega^\lambda_{(2,1^{n-2})}u$ for all $\lambda\vdash n$ and $u\in I^\lambda$. This implies that, similarly to Lemma~\ref{lem:necessary}, if $(T_n^*,Y)$ is a tiling of $S_n$ with $n\geq3$, then $c^\lambda\overline{Y}=0$ for each $\lambda\vdash n$ with $\lambda\neq (n)$ and $\omega^\lambda_{(2,1^{n-2})}\neq0$. Then by the argument in the proof of Theorem~\ref{main-theorem1} and Corollary~\ref{main-corollary1}, one proves Theorem~\ref{main-theorem2} and Corollary~\ref{main-corollary2}.
\qed

\section{Conjectures and discussions}\label{sec:conjecture}

In the proceeding sections we already see that, in a tiling $(T_n,Y)$ of $S_n$, the subset $Y$ is restricted in size, transitivity, scope where $\overline{Y}$ belongs to (see Lemma \ref{lem:necessary}). Some other requirements will also be revealed later in this section. Based on this information, we believe that such a subset $Y$ can hardly exist, and thus the following conjecture seems reasonable.

\begin{conjecture}\label{conj:conj1}
For $n\geq4$, neither $T_n$ nor $T_n^*$ tiles $S_n$.
\end{conjecture}


It can be directly verified for $n=4$ that neither $T_n$ nor $T_n^*$ tiles $S_n$. Thus, assume $$n\geq5$$ in the following discussion. By Corollary~\ref{corol:corol2} and Theorem~\ref{main-theorem2}, if either $T_n$ or $T_n^*$ tiles $S_n$, then there exists some $\lfloor(n-2)/2\rfloor$-transitive subset $Y$ of $S_n$ such that $Y$ is not $A_n$ or $S_n$. Given any positive integers $\lambda_2\geq\dots\geq\lambda_\ell$, Martin and Sagan~\cite[Section~6]{MS2006} constructed infinitely many values for $\lambda_1\geq\lambda_2$ such that there exists a $\lambda$-transitive subset of $S_n$ with $\lambda=(\lambda_1,\lambda_2,\ldots,\lambda_\ell)$. For each positive integer $k$, by taking $\ell=k+1$ and $\lambda_2=\cdots=\lambda_\ell=1$ we see that there exist $k$-transitive subsets of $S_n$ for infinitely many values of $n$.
Since this only gives examples of $k$-transitive subsets of $S_n$ with $k/n\to0$ whereas $\lfloor(n-2)/2\rfloor/n\to1/2$, we pose the following conjecture, which will imply Conjecture~\ref{conj:conj1} for sufficiently large $n$.

\begin{conjecture}\label{conj:conj3}
For sufficiently large $n$, there is no $\lfloor(n-2)/2\rfloor$-transitive subset of $S_n$ other than $A_n$ and $S_n$.
\end{conjecture}

In the rest of this section we focus on Conjecture~\ref{conj:conj1}, providing more insights into this conjecture. Denote $T_n^2:=\{t_1t_2\mid t_1,t_2\in T_n\}$ and $(T_n^*)^2:=\{t_1t_2\mid t_1,t_2\in T_n^*\}$. Then
\[
T_n^2\setminus\{\id\}=\CC_{(2,1^{n-2})}\cup\CC_{(3,1^{n-3})}\cup\CC_{(2^2,1^{n-4})}\ \text{ and }\ (T_n^*)^2\setminus\{\id\}=\CC_{(3,1^{n-3})}\cup\CC_{(2^2,1^{n-4})}.
\]
As usual, $\alpha(\Gamma)$ denotes the independence number of a graph $\Gamma$, namely, the maximum size of an independent set (coclique) of $\Gamma$. Let us start with the following observation.

\begin{lemma}\label{lem:independent}
Let $\Sigma_n:=\Cay(S_n,T_n^2\setminus\{\id\})$ and $\Sigma_n^*:=\Cay(S_n,(T_n^*)^2\setminus\{\id\})$. Then the following statements hold:
\begin{enumerate}[{\rm(a)}]
\item $(T_n,Y)$ is a tiling of $S_n$ if and only if $Y$ is an independent set of size $n!/(1+n(n-1)/2)$ in $\Sigma_n$.
\item $(T_n^*,Y)$ is a tiling of $S_n$ if and only if $Y$ is an independent set of size $2(n-2)!$ in $\Sigma_n^*$.
\item $T_n$ tiles $S_n$ if and only if $\alpha(\Sigma_n)\geq n!/(1+n(n-1)/2)$.
\item $T_n$ tiles $S_n$ if and only if $\alpha(\Sigma_n)=n!/(1+n(n-1)/2)$.
\item $T_n^*$ tiles $S_n$ if and only if $\alpha(\Sigma_n^*)\geq2(n-2)!$.
\item $T_n^*$ tiles $S_n$ if and only if $\alpha(\Sigma_n)=2(n-2)!$.
\end{enumerate}
\end{lemma}

\begin{proof}
By the definition of a tiling, $(T_n,Y)$ is a tiling of $S_n$ if and only if $|Y|=n!/|T_n|$ and there do not exist distinct $y_1$ and $y_2$ in $Y$ and distinct $t_1$ and $t_2$ in $T_n$ such that $t_1y_1=t_2y_2$. Notice that $|T_n|=1+n(n-1)/2$
and
\begin{equation}\label{equ:9}
t_1y_1=t_2y_2\;\Leftrightarrow\;y_1y_2^{-1}=t_1^{-1}t_2\;\Leftrightarrow\;y_1y_2^{-1}=t_1t_2.
\end{equation}
We conclude that $(T_n,Y)$ is a tiling of $S_n$ if and only if $|Y|=n!/(1+n(n-1)/2)$ and $Y$ is an independent set of $\Cay(S_n,T_n^2\setminus\{\id\})=\Sigma_n$. This proves statement~(a), and a similar argument gives a proof of statement~(b). Note that~\eqref{equ:9} also shows $T_ny_1\cap T_ny_2=\varnothing$ for any distinct $y_1$ and $y_2$ from an independent set of $\Sigma_n$. Hence $\alpha(\Sigma_n)\leq|S_n|/|T_n|=n!/(1+n(n-1)/2)$, and statements~(c) and~(d) follow from~(a). Similarly, statements~(e) and~(f) are consequences of~(b).
\qed
\end{proof}

\subsection{Fourier transform of Boolean functions on $S_n$}

For a group $G$ and a complex-valued function $f$ on $G$, the {\em Fourier transform} of $f$ is the matrix-valued function $\widehat{f}$ on the set of irreducible representations of $G$ such that
\[
\widehat{f}(\rho):=\frac{1}{|G|}\sum_{g\in G}f(g)\rho(g)
\]
for each irreducible representation $\rho$ of $G$. In particular, for a subset $Y$ of $S_n$, the Fourier transform of the characteristic function $\mathrm{1}_Y$ satisfies
\[
\widehat{1_Y}(\rho)=\frac{1}{n!}\sum_{g\in S_n}1_Y(g)\rho(g)=\frac{1}{n!}\sum_{g\in Y}\rho(g).
\]
Note that the only partitions $\lambda\vdash n$ that do not satisfy $\lambda\unrhd(3, 1^{n-3})$ are $\lambda=(2^k,1^{n-2k})$ with nonnegative $k$, which are such that $\sum_{x\in D(\lambda)}\xi(x)<-1$ as $n\geq5$.
If $(T_n, Y)$ is a tiling of $S_n$, then we derive from Lemmas~\ref{lem:cencha} and~\ref{lem:necessary} that $\widehat{1_Y}$ is supported only on irreducible representations $S^\lambda$ with $\lambda\unrhd(3, 1^{n-3})$. Similarly, if $(T_n^*, Y)$ is a tiling of $S_n$, we also have that $\widehat{1_Y}$ is supported only on irreducible representations $S^\lambda$ with $\lambda\unrhd(3, 1^{n-3})$.

As a statement concerning Boolean functions on $S_n$ and their Fourier transforms,~\cite[Theorem~27]{EFP2011} would imply Conjecture~\ref{conj:conj1}. In fact, based on the above conclusion,~\cite[Theorem~27]{EFP2011} asserts that the set $Y$ in any tiling $(T_n,Y)$ or $(T_n^*,Y)$ is a disjoint union of left cosets of stabilizers in $S_n$ of $n-3$ points.
Then by Lemma~\ref{lem:tiling}, the set $Y^{-1}$ in any tiling $(T_n,Y)$ or $(T_n^*,Y)$ is a disjoint union of left cosets of stabilizers in $S_n$ of $n-3$ points.
In this case, for any $y\in Y$, there exists a $3$-cycle $(i,j,k)$ in $S_n$ such that $y^{-1}(i,j,k)\in Y^{-1}$, which implies that $(i,k,j)y\in Y$, contradicting Lemma~\ref{lem:independent}.

Unfortunately,~\cite[Theorem~27]{EFP2011} is found to be false \cite{EFP2017,Filmus2017}.

\subsection{Eigenvalues of Cayley graphs}

Recall from the concept of an $r$-tiling (see Subsection~\ref{subsec:background}) that a $1$-tiling of a group $G$ is just a tiling of $G$. The following lemma gives a necessary condition for $r$-tilings in terms of eigenvalues of Cayley digraphs.

\begin{lemma}[{\cite[Lemma~1]{Terada2004}}]\label{lem:eigen}
If $(X,Y)$ is an $r$-tiling of a group $G$ for some positive integer $r$ such that $Y\neq G$, then $0$ is an eigenvalue of $\Cay(G,X)$.
\end{lemma}

Based on Lemma~\ref{lem:decomp}, it is easy to know (refer to~\cite{Zieschang1988}) the eigenvalues and eigenvectors of a Cayley digraph $\Cay(G,S)$ with $S$ normal (closed under conjugation) in $G$.

\begin{lemma}\label{lem:cayley}
Let $G$ be a group of order $m$. Then there exist pairwise orthogonal vectors $v_1,\dots,v_m\in\mathbb{C}^m$ such that, for each normal subset $S$ of $G$, the adjacency matrix of $\Cay(G,S)$ has eigenvectors $v_1,\dots,v_m\in\mathbb{C}^m$ with eigenvalues
\[
\frac{1}{\chi(\mathrm{id})}\sum_{g\in S}\chi(g),
\]
where $\chi$ runs over the irreducible characters of $G$.
\end{lemma}

Then by Lemma~\ref{lem:cencha}, the eigenvalues of $\Cay(S_n,T_n)$ are $1+\sum_{x\in D(\lambda)}\xi(x)$ with partitions $\lambda$ of $n$. This result in conjunction with Lemma~\ref{lem:eigen} can be used to show the nonexistence of tilings $(T_n,Y)$ of $S_n$ for some small values of $n$. For example, there is no partition $\lambda\vdash6$ with content sum $\sum_{x\in D(\lambda)}\xi(x)=-1$, and so $0$ is not an eigenvalue of $\Cay(S_6,T_6)$, which implies the nonexistence of tilings $(T_6,Y)$ of $S_6$ (note that this nonexistence result cannot be obtained from Lemma~\ref{main-corollary1} or~\ref{corol:corol1}).
For a general $n$, however, this approach does not work, as there always exists $\lambda\vdash n$ with $\sum_{x\in D(\lambda)}\xi(x)=-1$ whenever $n\geq14$. Similarly, since $n$ always has a partition with content sum $0$, Lemma~\ref{lem:eigen} does not help to exclude the tilings $(T_n^*,Y)$ of $S_n$.

Recall from Lemma \ref{lem:independent} the definitions of $\Sigma_n$ and $\Sigma_n^*$. Inspired by Lemma~\ref{lem:independent}, we now turn to the independence numbers of $\Sigma_n$ and $\Sigma_n^*$, which are well known to be related to the eigenvalues of these Cayley graphs, for instance, by the following result~\cite[Theorem~12]{EFP2011}.

\begin{lemma}[Weighted version of Hoffman bound]\label{prop:weighted}
Let $\Gamma$ be a graph on $m$ vertices, let $\Gamma_j$ be a $d_j$-regular spanning subgraph of $\Gamma$ for $j\in\{1,\ldots,r\}$ such that all $\Gamma_j$ share a common orthonormal system of eigenvectors $v_1,\ldots,v_m$, and let $\beta_1,\ldots,\beta_r\in\mathbb{R}$. For each $i\in\{1,\ldots,m\}$, let $\ell_i:=\sum_{j=1}^r\beta_j\ell_{ij}$, where $\ell_{ij}$ is the eigenvalue of $v_i$ in $\Gamma_j$. If $Y$ is an independent set of $\Gamma$, then setting $d:=\sum_{j=1}^r\beta_j d_j$ and $\ell_{\mathrm{min}}:=\min_{i\in\{1,\ldots,m\}}\ell_i$, we have
\[
\frac{|Y|}{m}\leq\frac{-\ell_{\mathrm{min}}}{d-\ell_{\min}};
\]
moreover, the equality holds only if the characteristic vector $1_Y$ of $Y$ is a linear combination of the all-$1$'s vector and those $v_i$ with $\ell_i=\ell_{\min}$.
\end{lemma}

By Lemma~\ref{lem:cayley}, if we take $\Gamma=\Sigma_n$ in Lemma~\ref{prop:weighted} with
\[
\Gamma_1=\Cay\big(S_n,\CC_{(2,1^{n-2})}\big),\ \ \Gamma_2=\Cay\big(S_n,\CC_{(3,1^{n-3})}\big),\ \ \Gamma_3=\Cay\big(S_n,\CC_{(2^2,1^{n-4})}\big),
\]
then the $\ell_i$'s turn out to be
\[
\ell_\lambda:=\beta_1\omega^{\lambda}_{(2,1^{n-2})}+\beta_2\omega^{\lambda}_{(3,1^{n-3})}+\beta_3\omega^{\lambda}_{(2^2,1^{n-4})}
\]
with $\lambda$ running over the partitions of $n$. According to~\cite[Pages~150~and~152]{Katriel996},
\begin{align*}
  \omega^{\lambda}_{(2,1^{n-2})}&=\sum_{x\in D(\lambda)}\xi(x), \\
  \omega^{\lambda}_{(3,1^{n-3})}&=\sum_{x\in D(\lambda)}\xi(x)^2-\dfrac{n(n-1)}{2}, \\
  \omega^{\lambda}_{(2^2,1^{n-4})}&=\dfrac{1}{2}\Big(\sum_{x\in D(\lambda)}\xi(x)\Big)^2-\dfrac{3}{2}\sum_{x\in D(\lambda)}\xi(x)^2 +\dfrac{n(n-1)}{2}.
\end{align*}
Denote $d_1:=\big|\CC_{(2,1^{n-2})}\big|$, $d_2:=\big|\CC_{(3,1^{n-3})}\big|$ and $d_3:=\big|\CC_{(2^2,1^{n-4})}\big|$. Then
\[
d_1=\frac{n(n-1)}{2},\ \ d_2=\frac{n(n-1)(n-2)}{3},\ \ d_3=\frac{n(n-1)(n-2)(n-3)}{8}.
\]
Suppose that $(T_n,Y)$ is a tiling of $S_n$. Then a combination of Lemmas~\ref{lem:independent} and~\ref{prop:weighted} gives
\begin{equation}\label{equ:8}
\frac{1}{|T_n|}=\frac{|Y|}{n!}\leq\frac{-\ell_{\min}}{\beta_1d_1+\beta_2d_2+\beta_3d_3-\ell_{\min}}
\end{equation}
for any $\beta_1,\beta_2,\beta_3\in\mathbb{R}$, where $\ell_{\min}$ is the minimum of $\ell_\lambda$ for $\lambda\vdash n$.
For $(\beta_1,\beta_2,\beta_3)=(2,3,2)$, the minimum of $\ell_{\min}$ is taken when $\sum_{x\in D(\lambda)}\xi(x)=-1$, and so~\eqref{equ:8} turns out to be
\[
\frac{1}{|T_n|}\leq\frac{\frac{n(n-1)}{2}+1}{2d_1+3d_2+2d_3+\frac{n(n-1)}{2}+1}=\frac{2}{n^2-n+2}.
\]
However, this is a tight bound for $|Y|/n!$ instead of a contradiction, and the ``moreover'' part in Lemma~\ref{prop:weighted} only yields the conclusion of Lemma~\ref{lem:necessary}. Similarly, if $(T_n^*,Y)$ is a tiling of $S_n$, then taking $(\beta_1,\beta_2,\beta_3)=(0,3,2)$ we obtain $\ell_{\min}$ when $\sum_{x\in D(\lambda)}\xi(x)=0$, which gives the tight bound
\[
\frac{1}{|T_n^*|}\leq\frac{\frac{n(n-1)}{2}}{3d_2+2d_3+\frac{n(n-1)}{2}}=\frac{2}{n^2-n}
\]
for $|Y|/n!$. Currently, we do not know whether it works to prove Conjecture~\ref{conj:conj1} by taking certain values of $\beta_1,\beta_2,\beta_3\in\mathbb{R}$, possibly depending on $n$.

\subsection{Forbidden intersection problem for permutations}

For $i\in\{1,\ldots,n\}$, a subset $Y$ of $S_n$ is said to {\em avoid intersection $i$} if the number of fixed points of $y_1y_2^{-1}$ is not equal to $i$ for any $y_1,y_2\in Y$. By Lemma~\ref{lem:independent}, the set $Y$ in any tiling $(T_n,Y)$ or $(T_n^*,Y)$ of $S_n$ avoids intersection $n-3$. Thus, a necessary condition for $T_n$ or $T_n^*$ to tile $S_n$ is the existence of a set of size $n!/(1+n(n-1)/2)$ or $2(n-2)!$, respectively, in $S_n$ that avoids intersection $n-3$.
The forbidden intersection problem for permutations concerns the maximum size for a subset of $S_n$ that avoids intersection $i$, which has attracted considerable attention with breakthroughs achieved recently~\cite{EL2022,KL2017,KZ2024}. However, existing results pertain only to relatively small $i$ (for example, $i=O(\sqrt{n}/\log_2n)$~\cite[Theorem~2]{KZ2024}), while the results we need for Conjecture~\ref{conj:conj1} is $i=n-3$, as posed in the following question.

\begin{question}
What is the maximum size for a subset of $S_n$ that avoids intersection $n-3$?
\end{question}

\section*{Acknowledgements}
The authors would like to thank the anonymous referee for helpful comments that have improved the presentation. We also wish to thank Professors William Yong-Chuan Chen and Xin-Gui Fang for their advice and inspiration. Teng Fang was supported by the Scientific Research Foundation for Advanced Talents of Suqian University (no.~2025XRC027). Binzhou Xia acknowledges the support of ARC Discovery Project DP250104965.

\end{document}